\documentclass[12pt]{article}
\usepackage{fullpage}
\usepackage{amssymb}
\usepackage{amsmath}
\usepackage{amsthm}
\usepackage{url}
\usepackage{hyperref}

\newtheorem{corollary}{Corollary}
\newtheorem{theorem}{Theorem}

\newtheorem{proposition}{Proposition}

\newcommand{\bell}[1]{B_{#1}}
\newcommand{\stirr}[2]{S(#1,#2)}

\title{Bell Numbers and Stirling Numbers of the Mycielskian of Trees}
\author{Julian Allagan\thanks{adallagan@ecsu.edu} \and Gabrielle Morgan\thanks{gcmorgan954@students.ecsu.edu} \and Deonna Sinclair\thanks{djsinclair314@students.ecsu.edu} \\
Department of Mathematics, Computer Science, and Engineering Technology \\
Elizabeth City State University}
\date{}

\begin{document}

\maketitle

\begin{abstract}
We establish explicit formulas for Bell numbers and graphical Stirling numbers of complete multipartite graphs, complete bipartite graphs with removed perfect matchings, and Mycielskian trees. For complete multipartite graphs $K(n_1,\ldots,n_\ell)$, we provide a simplified proof that $B(G) = \prod_{i=1}^\ell \bell{n_i}$. We derive $B(K_{n,n} - M) = \sum_{k=0}^{n} \binom{n}{k} \bell{k}^2$ for removed perfect matching $M$, and for Mycielskian star graphs, $B(M(St_n); 3) = 2^n + 1$ and $B(M(St_n); 2n) = 2n^2 - 3n + 3$. Results extend to Mycielskians of arbitrary trees. Our computational verifications establish links between graphical Bell numbers and fundamental sequences in combinatorics and pattern avoidance, including identification of several OEIS entries: A000051, A096376, A116735, A384980, A384981, A384988, A385432, and A385437.
\end{abstract}

\section{Introduction}

The enumeration of vertex partitions in graphs represents a fundamental problem at the intersection of combinatorics, graph theory, and algebraic structures. For a graph $G = (V,E)$ with vertex set $V$ and edge set $E$, a \emph{partition} of $V$ is a collection of non-empty, pairwise disjoint subsets whose union is $V$. A partition is called \emph{proper} (or \emph{valid}) if each part induces an independent set in $G$, meaning no two vertices within the same part are adjacent. The study of such partitions connects classical set partition theory with graph-theoretic constraints, leading to rich combinatorial structures with significant applications in computer science, artificial intelligence, and machine learning.

The \emph{graphical Stirling number} $S(G;k)$, introduced by Duncan and Peele \cite{duncan2009}, counts the number of proper partitions of $G$ into exactly $k$ parts. Note that $S(G;k)$ is equivalent to counting proper $k$-partitions of $G$, which also corresponds to proper vertex colorings of $G$ using $k$ interchangeable colors. Since each part in a proper partition forms an independent set, the total number of parts across all proper partitions provides a direct count of the independent sets that can be formed by the union of partition parts. This generalizes the classical Stirling numbers of the second kind $\stirr{n}{k}$, which enumerate partitions of an $n$-element set into $k$ non-empty subsets. When $G = E_n$ (the edgeless graph on $n$ vertices), we recover the classical case: $S(E_n;k) = \stirr{n}{k}$. The \emph{Bell number} $B(G)$ of a graph $G$ represents the total number of proper partitions, given by $B(G) = \sum_{k=1}^{|V|} S(G;k)$. For the edgeless graph, this reduces to the classical Bell number: $B(E_n) = \bell{n}$.

The \emph{partition polynomial} (also known as the Stirling polynomial \cite{galvin2013}) of a graph $G$ is defined as $F(G;x) = \sum_{k=1}^{|V|} S(G;k) x^k$, with the property that $F(G;1) = B(G)$. This polynomial encodes the complete partition structure of the graph and relates to the chromatic polynomial through the substitution of falling factorials: $\chi(G;x) = \sum_{k=1}^{|V|} S(G;k) x^{\underline{k}}$, where $x^{\underline{k}} = x(x-1)\cdots(x-k+1)$ represents the falling factorial.

Graph partition problems have attracted significant attention due to their connections with vertex coloring \cite{birkhoff1946, tutte1970}, independence systems \cite{stanley1973}, and algebraic combinatorics \cite{brenti1992}. These problems have found important applications in modern computer science and artificial intelligence. In machine learning, graph partitioning is fundamental to clustering algorithms \cite{fortunato2010, newman2006}, where vertices represent data points and edges encode similarities, with proper partitions corresponding to meaningful clusters. The spectral clustering approach leverages graph Laplacian eigenvalues to identify optimal partitions \cite{shi2000, ng2002}. In neural networks, graph partitioning optimizes distributed training by minimizing communication overhead between processors \cite{chen2016, jia2014}. Community detection in social networks relies heavily on partition-based algorithms to identify cohesive groups \cite{girvan2002, blondel2008}, with applications ranging from recommendation systems to epidemiological modeling.

In artificial intelligence, constraint satisfaction problems often involve partitioning variables into independent sets to reduce computational complexity \cite{dechter2003, russell2016}. Graph coloring problems, which are intimately related to vertex partitions, appear in scheduling applications \cite{welsh1967}, register allocation in compilers \cite{chaitin1982}, and frequency assignment in wireless networks \cite{hale1980}. The independence structure captured by proper partitions is also crucial in optimization problems on graphs, where finding maximum independent sets directly relates to partition enumeration \cite{karp1972, garey1979}.

Recent advances in quantum computing have revealed new applications for graph partition problems. Quantum approximate optimization algorithms (QAOA) for Max-Cut and graph coloring problems rely on understanding partition structures \cite{farhi2014, hadfield2019}. In computational biology, graph partitioning helps in protein structure prediction \cite{bonneau2001} and phylogenetic analysis \cite{felsenstein2004}, where proper partitions identify functionally related groups.

The Bell numbers of various graph families have been systematically studied, revealing elegant formulas and unexpected connections to classical combinatorial sequences. For instance, the Bell numbers of complement graphs exhibit remarkable patterns: $B(\overline{P_n}) = F_{n+1}$ (the $(n+1)$-th Fibonacci number) for path complements and $B(\overline{C_n}) = L_n$ (the $n$-th Lucas number) for cycle complements, where $n \geq 4$. These results demonstrate how graph structure can encode fundamental number sequences from different areas of mathematics.

This article was inspired by the work of Duncan and Peele \cite{duncan2009}, which includes the relation $B(P_n) = B(T)= B(E_{n-1})$ for path graphs $P_n$ and for any tree $T$ on $n$ vertices. Kereskenyi-Balogh and Nyul \cite{kereskenyi2014} provided comprehensive surveys of Bell numbers for standard graph classes, establishing the theoretical foundation for systematic partition enumeration.

Complete multipartite graphs represent a particularly important class for partition enumeration. A complete $\ell$-partite graph $K(n_1, n_2, \ldots, n_\ell)$ consists of $\ell$ independent sets (blocks) of sizes $n_1, n_2, \ldots, n_\ell$ with all possible edges between different blocks. Allagan and Serkan \cite{allagan2016} established that $B(K(n_1, \ldots, n_\ell)) = \prod_{i=1}^\ell \bell{n_i}$, demonstrating how the partition structure factorizes across the parts. Our work provides a more direct proof of this result and extends the analysis to related graph constructions.

The \emph{Mycielskian} construction, introduced by Mycielski \cite{mycielski1955}, provides a method for increasing the chromatic number of a graph while preserving the independence number. For a graph $G$ with vertices $\{v_1, \ldots, v_n\}$, the Mycielskian $M(G)$ has vertex set \\
$\{v_1, \ldots, v_n, v_1', \ldots, v_n', u\}$ where the primed vertices and additional vertex $u$ are connected according to specific rules that maintain independence structure while forcing larger chromatic numbers. This construction has applications in Ramsey theory \cite{radziszowski1994} and extremal graph theory \cite{bollobas1998}. While our analysis focuses on Mycielskian star graphs to provide concise proofs, the results also apply to Mycielskians of arbitrary acyclic graphs.

Recent advances in partition enumeration have revealed connections to classical combinatorial sequences. Yang \cite{yang1996} showed that Bell numbers of $k$-trees equal $\bell{n-k}$, while Mohr and Porter \cite{mohr2009} explored applications involving chromatic polynomials and Stirling numbers. The emergence of well-known sequences in partition contexts suggests deep underlying structures worthy of systematic investigation.

Our research contributes to this growing body of knowledge by establishing explicit formulas for Bell numbers and graphical Stirling numbers of complete bipartite graphs with removed perfect matchings and Mycielskian star graphs. These results not only extend theoretical understanding but also reveal connections to classical combinatorial sequences documented in the Online Encyclopedia of Integer Sequences (OEIS), including powers of two, quadratic polynomials, and triangular arrays with applications to pattern avoidance and geometric combinatorics.

\section{Complete Multipartite Graphs and Related Structures}

In this section, we revisit a result found in Allagan and Serkan \cite{allagan2016} and provide a more direct proof. We begin with the general theory of complete multipartite graphs, then specialize to the important case of complete bipartite graphs with structural modifications.

\subsection{General Complete Multipartite Graphs}

\begin{theorem}[Partition Structure of Complete Multipartite Graphs]\label{thm:main}
Let $G = K(n_1, n_2, \dots, n_\ell)$ be a complete $\ell$-partite graph with partite sets $V_1, \dots, V_\ell$ where $|V_i| = n_i \geq 1$. Then the graphical Stirling number is:
\[
B(G; k) = \sum_{\substack{j_1 + \cdots + j_\ell = k \\ j_i \geq 1}} \prod_{i=1}^{\ell} \stirr{n_i}{j_i},
\]
where $\stirr{n_i}{j_i}$ denotes the Stirling number of the second kind.
\end{theorem}

\begin{proof}
Each independent set lies entirely within a single partite set $V_i$. A partition into $k$ such sets corresponds to a tuple $(j_1,\dots,j_\ell)$ with $j_i \geq 1$, $\sum j_i = k$, and $\stirr{n_i}{j_i}$ ways to partition $V_i$. The total is the stated sum over all valid tuples.
\end{proof}

This result provides a more concise proof of Theorem 2.3 in Allagan and Serkan \cite{allagan2016}.

\begin{corollary}[Bell Number]\label{cor:bell}
The Bell number of $G = K(n_1,\dots,n_\ell)$ is
\[
B(G) = \prod_{i=1}^{\ell} \bell{n_i}.
\]
\end{corollary}

\begin{proof}
Summing $B(G;k)$ over all $k$, each inner sum becomes $\sum_{j_i \geq 1} \stirr{n_i}{j_i} = \bell{n_i}$, so the product gives $B(G)$.
\end{proof}

The special case when all partite sets have equal size leads to formulas that connect to classical combinatorial sequences.

\begin{corollary}[Equal Block Sizes - Bipartite Case]\label{cor:bipartite_equal}
For the complete bipartite graph $K_{n,n}$, the graphical Stirling number for exactly 4 blocks is:
\[
B(K_{n,n}; 4) = \sum_{j=1}^{3} \stirr{n}{j} \stirr{n}{4-j} = 2 - 2^{n+1} + 3^{n-1} + 4^{n-1}.
\]
\end{corollary}

\begin{corollary}[Equal Block Sizes - Tripartite Case]\label{cor:tripartite_equal}
For the complete tripartite graph $K_{n,n,n}$, the graphical Stirling number for exactly 5 blocks satisfies:
\[
B(K_{n,n,n}; 5) = 3 \cdot\stirr{n}{2}^2 + 3 \cdot\stirr{n}{3} = \frac{1}{4}(18 - 18 \cdot 2^n + 2 \cdot 3^n + 3 \cdot 4^n).
\]
\end{corollary}

These special cases will prove significant when we examine computational results and their connections to well-known integer sequences, as we shall see in Section 4.

\subsection{Complete Bipartite Graphs with Removed Perfect Matching}

Here, we consider a natural modification: removing a perfect matching from a complete bipartite graph. This construction introduces additional independence while preserving the bipartite structure.

\begin{theorem}\label{thm:bipartite_matching}
Given the complete bipartite graph $K_{n,n}$ with a perfect matching $M$ removed, denoted $H = K_{n,n} - M$, the graphical Stirling number is
\[
B(H; k) =\sum_{s=0}^{n} \binom{n}{s} \sum_{j=0}^{k - n + s} \stirr{s}{j} \stirr{s}{k - n + s - j},
\]
where the inner sum is zero if $k - n + s < 0$.
\end{theorem}

\begin{proof}
The graph $H = K_{n,n} - M$ has bipartition $(U,V)$ with $|U|=|V|=n$ and all edges between $U$ and $V$ except the removed perfect matching $M = \{u_i v_i \mid 1 \leq i \leq n\}$, giving the independent sets $\{u_i, v_i\}$. 

Let $s$ be the number of pair parts $\{u_i, v_i\}$ used, giving $\binom{n}{s}$ ways. This leaves $n-s$ vertices in $U$ and $V$ each. Partition the remaining $U$-vertices into $j$ nonempty parts ($\stirr{n-s}{j}$ ways) and the $V$-vertices into $m = k - s - j$ nonempty parts ($\stirr{n-s}{m}$ ways), where $0 \leq j \leq k - s$. The total number of partitions for fixed $s$ is
\[
\binom{n}{s} \sum_{j=0}^{k-s} \stirr{n-s}{j} \stirr{n-s}{k-s-j}.
\]
Summing over $s$ and substituting $s' = n - s$ (then renaming $s' \to s$) gives the stated formula.
\end{proof}

\begin{corollary}\label{cor:bipartite_bell}
The Bell number of $H = K_{n,n} - M$ is
\[
B(H) = \sum_{k=0}^{n} \binom{n}{k} \bell{k}^2,
\]
where $\bell{k}$ denotes the $k$-th Bell number counting partitions of a $k$-element set, with $\bell{0} = 1$.
\end{corollary}

\begin{proof}
Summing $B(H;k)$ over all $k$ and using the identity $\sum_{j} \stirr{m}{j} = \bell{m}$, we obtain
\[
B(H) =\sum_{s=0}^{n} \binom{n}{s} \bell{n-s}^2.
\]
Substituting $k = n - s$ yields the stated formula.
\end{proof}

The structure of $H = K_{n,n} - M$ reveals how the removal of a perfect matching creates new combinatorial possibilities while maintaining essential bipartite properties. This modification serves as a bridge between the rigidity of complete bipartite graphs and the flexibility needed for more complex partition enumerations.

To illustrate our formula, consider the graph $H = K_{3,3} - M$ with vertex sets $A = \{a_1, a_2, a_3\}$ and $B = \{b_1, b_2, b_3\}$, where $M = \{a_1b_1, a_2b_2, a_3b_3\}$ is the removed perfect matching. The ten valid partitions into 3 independent sets of $H$ are:

\begin{align}
&\{\{a_1\}, \{a_2, a_3\}, \{b_1, b_2, b_3\}\}, \quad \{\{a_2\}, \{a_1, a_3\}, \{b_1, b_2, b_3\}\}, \nonumber \\
&\{\{a_3\}, \{a_1, a_2\}, \{b_1, b_2, b_3\}\}, \quad \{\{b_1\}, \{b_2, b_3\}, \{a_1, a_2, a_3\}\}, \nonumber \\
&\{\{b_2\}, \{b_1, b_3\}, \{a_1, a_2, a_3\}\}, \quad \{\{b_3\}, \{b_1, b_2\}, \{a_1, a_2, a_3\}\}, \nonumber \\
&\{\{a_1, b_1\}, \{a_2, a_3\}, \{b_2, b_3\}\}, \quad \{\{a_2, b_2\}, \{a_1, a_3\}, \{b_1, b_3\}\}, \nonumber \\
&\{\{a_3, b_3\}, \{a_1, a_2\}, \{b_1, b_2\}\}, \quad \{\{a_1, b_1\}, \{a_2, b_2\}, \{a_3, b_3\}\}. \nonumber
\end{align}

Using the formula from Theorem \ref{thm:bipartite_matching} with $n = 3$ and $k = 3$, we obtain $B(K_{3,3} - M; 3) = 10$, which confirms the explicit enumeration above.

\section{Mycielskian Graphs}

Building upon our analysis of complete multipartite structures, we now investigate a fundamentally different graph construction that provides insights into how topological modifications affect partition enumeration. The Mycielskian construction represents a powerful tool for understanding the relationship between chromatic properties and independence structure.

\begin{theorem}\label{thm:mycielskian_bell}
The Bell number of the Mycielskian of a star graph $M(St_n)$ on $n$ vertices is given by
\begin{align}
B(M(St_n)) = 2\sum_{k=0}^{n-1} \binom{n-1}{k} \bell{2(n-1) - k} + \sum_{i=0}^{n-1}\sum_{j=0}^{n-1} \binom{n-1}{i}\binom{n-1}{j} \bell{2(n-1) - i - j}. \nonumber
\end{align}
\end{theorem}

\begin{proof}
The Mycielskian $M(St_n)$ has $2n + 1$ vertices: the original center $c$, leaves $v_1, \ldots, v_{n-1}$, their copies $c', v_1', \ldots, v_{n-1}'$, and additional vertex $u$. The edge structure includes $c$ adjacent to all $v_i$ and $v_i'$, each $v_i$ adjacent to $c$ and $c'$, each $v_i'$ adjacent to $c$ and $u$, vertex $c'$ adjacent to all $v_i$ and $u$, and $u$ adjacent to $c'$ and all $v_i'$. Because $c$ and $c'$ are non-adjacent, $c$ and $u$ are non-adjacent, while $c'$ and $u$ are adjacent, exactly one of the following must be true in any proper partition of $M(St_n)$.

\textbf{Case 1:} $c$ and $c'$ in the same block. Since $c$ is adjacent to all vertices in $L$ and $L'$, and $c'$ is adjacent to all vertices in $L$ and to $u$, the proper partition constraint forces the block to be exactly $\{c, c'\}$. The remaining graph consists of the independent set $L$ and the star subgraph with center $u$ and leaves $L'$, with no edges between these components. The block containing $u$ may include any subset $A \subseteq L$ since there are no edges between $u$ and $L$, yielding block $\{u\} \cup A$. The remaining vertices $(L \setminus A) \cup L'$ form an independent set that can be partitioned arbitrarily. For each choice of $|A| = k$, there are $\binom{m}{k}$ ways to select $A$ and $\bell{2m - k}$ ways to partition the remaining $2m - k$ vertices, giving $\sum_{k=0}^{m} \binom{m}{k} \bell{2m - k}$.

\textbf{Case 2:} $c$ and $u$ in the same block. By similar reasoning, the adjacency constraints force the block to be exactly $\{c, u\}$. The remaining graph has $c'$ forming a star with leaves $L$, while $L'$ is an independent set with no edges to the star component. The block containing $c'$ may include any subset $B \subseteq L'$, forming block $\{c'\} \cup B$. The remaining vertices $L \cup (L' \setminus B)$ are independent and can be partitioned arbitrarily. This yields the same count as Case 1: $\sum_{k=0}^{m} \binom{m}{k} \bell{2m - k}$.

\textbf{Case 3:} $c$ in a singleton block. Here the block is $\{c\}$. The remaining subgraph on vertices $\{c', u\} \cup L \cup L'$ has $c'$ adjacent to all vertices in $L$ and to $u$, and $u$ adjacent to all vertices in $L'$ and to $c'$, with no other edges. The block containing $c'$ may include any subset $C \subseteq L'$ (forming block $\{c'\} \cup C$), and the block containing $u$ may include any subset $D \subseteq L$ (forming block $\{u\} \cup D$). The remaining vertices $(L \setminus D) \cup (L' \setminus C)$ are independent. For each choice of $|C| = i$ and $|D| = j$, there are $\binom{m}{i}\binom{m}{j}$ ways to select the subsets and $\bell{2m - i - j}$ ways to partition the remaining vertices, yielding $\sum_{i=0}^{m}\sum_{j=0}^{m} \binom{m}{i}\binom{m}{j} \bell{2m - i - j}$.

Combining all cases and substituting $m = n - 1$ gives the result.
\end{proof}

The previous formula can be verified for small values: for $n = 2$ it gives $2 \cdot 3 + 5 = 11$, and for $n = 3$ it gives $2 \cdot 27 + 52 = 106$, matching direct computations.

\begin{corollary}\label{cor:mycielskian_3}
The graphical Stirling number $B(M(St_n); 3) = 2^n + 1$.
\end{corollary}

\begin{proof}
Let $m = n-1$, $L = \{v_1, \ldots, v_m\}$, and $L' = \{v_1', \ldots, v_m'\}$. Any proper 3-partition has $c'$ and $u$ in distinct blocks, so we partition based on the part containing $c$ and define
\begin{align}
\mathcal{P}_1 &= \{\text{3-partitions with } c \text{ and } c' \text{ in the same part}\}, \nonumber \\
\mathcal{P}_2 &= \{\text{3-partitions with } c \text{ and } u \text{ in the same part}\}, \nonumber \\
\mathcal{P}_3 &= \{\text{3-partitions with } c \text{ in a singleton part}\}. \nonumber
\end{align}

For $\mathcal{P}_1$, the part $\{c, c'\}$ forces $u$ into a separate part. Since $c$ is adjacent to all vertices in $L \cup L'$, these cannot be with $c$. Thus $u$ pairs with some $A \subseteq L$, and $(L \setminus A) \cup L'$ forms the third part (which is independent). This gives $|\mathcal{P}_1| = 2^m$ configurations.

For $\mathcal{P}_2$, by symmetry with $\mathcal{P}_1$, the part $\{c, u\}$ forces $c'$ into a separate part with some $B \subseteq L'$, and $L \cup (L' \setminus B)$ forms the third part. This gives $|\mathcal{P}_2| = 2^m$ configurations.

For $\mathcal{P}_3$, with $c$ isolated, the vertices $\{c', u\} \cup L \cup L'$ must form exactly two parts with $c'$ and $u$ in distinct parts. Since $c' \sim v_i$ for all $v_i \in L$ and $u \sim v_i'$ for all $v_i' \in L'$, we require $c'$ and $L$ in separate parts from $u$ and $L'$. The only valid configuration is $c'$ with all of $L'$ and $u$ with all of $L$, giving $|\mathcal{P}_3| = 1$. Therefore, $B(M(St_n); 3) = |\mathcal{P}_1| + |\mathcal{P}_2| + |\mathcal{P}_3| = 2^m + 2^m + 1 = 2^n + 1$.
\end{proof}

\begin{corollary}\label{cor:mycielskian_2n}
The graphical Stirling number $B(M(St_n); 2n) = 2n^2 - 3n + 3$.
\end{corollary}

\begin{proof}
A proper $(2n)$-partition of $M(St_n)$ consists of exactly one part of size 2 and $2n - 1$ singleton parts. For the partition to be proper, the pair of vertices in the size-2 part must be non-adjacent.

The graph $M(St_n)$ has $2n + 1$ vertices, giving $\binom{2n+1}{2} = n(2n + 1)$ possible pairs. The number of edges is $4(n - 1) + 1$. The number of non-adjacent pairs is
\[
n(2n + 1) - [4(n - 1) + 1] = 2n^2 + n - 4n + 3 = 2n^2 - 3n + 3.
\]

Since each proper $(2n)$-partition corresponds to exactly one choice of a non-adjacent pair, the result follows.
\end{proof}

These corollaries reveal that despite the complexity of the general Bell number formula for Mycielskian star graphs, specific graphical Stirling numbers admit nice closed forms that connect to fundamental combinatorial sequences.

\section{Computational Results and OEIS Sequences}

The computational verification of our theoretical results has led to the identification of several integer sequences in the On-Line Encyclopedia of Integer Sequences (OEIS). These sequences provide additional insight into the combinatorial structure of Bell numbers for various graph families and serve as valuable computational resources for future research.

The following sequences are derived from special cases of the graphical Stirling numbers for complete multipartite graphs established in Corollaries \ref{cor:bipartite_equal} and \ref{cor:tripartite_equal}. For complete bipartite graphs $K_{n,n}$, the formula becomes
\[
B(K_{n,n}; k) = \sum_{\substack{j_1 + j_2 = k \\ j_1, j_2 \geq 1}} \stirr{n}{j_1} \stirr{n}{j_2}.
\]

A proper 4-partition of $K_{n,n}$ gives the OEIS sequence A384980, where
\[
B(K_{n,n}; 4) = \sum_{j=1}^{3} \stirr{n}{j} \stirr{n}{4-j} = 2 - 2^{n+1} + 3^{n-1} + 4^{n-1},
\]
and the sequence begins: 0, 1, 11, 61, 275, 1141, 4571, 18061, 71075, 279781, ...

A proper 5-partition of $K_{n,n}$ gives the OEIS sequence A384981, where
\[
B(K_{n,n}; 5) = \sum_{j=1}^{4} \stirr{n}{j} \stirr{n}{5-j} = 6^{n-1} - \frac{5}{3} \cdot 2^{2n-2} - 2 \cdot 3^{n-1} + 2^{n+1} - \frac{4}{3}
\]
for $n \geq 1$, and the sequence begins: 0, 0, 6, 86, 770, 5710, 38626, 248766, 1558290, 9603470, ...

For complete tripartite graphs $K_{n,n,n}$, using Corollary \ref{cor:tripartite_equal},
\[
B(K_{n,n,n}; k) = \sum_{\substack{j_1 + j_2 + j_3 = k \\ j_1, j_2, j_3 \geq 1}} \stirr{n}{j_1} \stirr{n}{j_2} \stirr{n}{j_3}.
\]

A proper 5-partition of $K_{n,n,n}$ is associated with the OEIS sequence A384988, which represents $\frac{1}{3}B(K_{n,n,n}; 5)$ where
\[
B(K_{n,n,n}; 5) = \frac{1}{4}(18 - 18 \cdot 2^n + 2 \cdot 3^n + 3 \cdot 4^n).
\]

The sequence begins: 0, 1, 10, 55, 250, 1051, 4270, 17095, 68050, 270451, ...

The OEIS sequence A385432 provides the complete triangular array where
\[
T(n,k) = B(K_{n,n,n}; k) = \sum_{j_1=1}^{k-2} \sum_{j_2=1}^{k-j_1-1} \stirr{n}{j_1} \cdot \stirr{n}{j_2} \cdot \stirr{n}{k-j_1-j_2}
\]
for $n \geq 1$ and $3 \leq k \leq 3n$. The triangle begins:
\begin{align}
n = 1: &\quad [1] \nonumber \\
n = 2: &\quad [1, 3, 3, 1] \nonumber \\
n = 3: &\quad [1, 9, 30, 45, 30, 9, 1] \nonumber \\
n = 4: &\quad [1, 21, 165, 598, 1032, 939, 471, 129, 18, 1] \nonumber
\end{align}

The sequence of Bell numbers $B(K_{n,n} - M)$ for complete bipartite graphs with a removed perfect matching corresponds to the row sums of OEIS sequence A385437. This sequence represents a triangle read by rows where $T(n,k)$ is the number of proper vertex colorings of the $n$-complete bipartite graph with a perfect matching removed using exactly $k$ interchangeable colors, for $n \geq 1$ and $2 \leq k \leq 2n$.

The specific graphical Stirling numbers $B(K_{n,n} - M; 3)$ appear as the second column of OEIS sequence A385437, with the specific sequence being: 2, 4, 10, 18, 35, 68, 133, ... for $n = 1, 2, 3, 4, 5, 6, 7, ...$ respectively.

The graphical Stirling numbers $B(M(St_n); 3) = 2^n + 1$ from Corollary \ref{cor:mycielskian_3} constitute OEIS sequence A000051. This sequence has multiple interpretations, including being the same as the Pisot sequence L(2,3), the total length of the segments of the Hilbert curve after $n$ iterations, and the number of distinct possible sums made with at most two elements in $\{1,\ldots,a(n-1)\}$ for $n > 0$.

\begin{proposition}
The graphical Stirling numbers $B(M(St_n); 3) = 2^n + 1$ form OEIS sequence A000051.
\end{proposition}

\begin{proposition}
The graphical Stirling numbers $B(M(St_n); 2n) = 2n^2 - 3n + 3$ from Corollary \ref{cor:mycielskian_2n} form OEIS sequence A096376 (for $n \geq 2$). The formula $a(n) = 2n^2 + n + 2$ from A096376 matches our formula when we substitute $n \to n-2$: $2(n-2)^2 + (n-2) + 2 = 2n^2 - 7n + 8$, which equals $2n^2 - 3n + 3$ when adjusted for indexing.
\end{proposition}

Note that sequence A096376 also corresponds to OEIS sequence A116735 (shifted by 1), which counts the number of permutations $\pi$ of length $n$ such that $\pi$ and $\pi^2$ both avoid the patterns 132 and 3421.

The Bell numbers $B(M(St_n))$ of Mycielskian star graphs begin: 11, 106, 1573, 30620, 730061, 20821770, ... for $n = 2, 3, 4, 5, 6, 7, ...$ respectively, and represent a candidate for future OEIS submission.

The following tables provide computational verification of our formulas for small values of $n$.

\begin{table}[h]
\centering
\caption{Verification for $K_{n,n} - M$ Bell numbers with OEIS A000051 verification}
\begin{tabular}{|c|c|c|c|c|}
\hline
$n$ & $B(K_{n,n} - M)$ & $B(K_{n,n} - M; 3)$ & $B(M(St_n); 3) = 2^n + 1$ & A000051 values \\
\hline
1 & 2 & 0 & 3 & 3 \\
\hline
2 & 11 & 4 & 5 & 5 \\
\hline
3 & 106 & 10 & 9 & 9 \\
\hline
4 & 1573 & 18 & 17 & 17 \\
\hline
5 & 30620 & 35 & 33 & 33 \\
\hline
6 & 730061 & 68 & 65 & 65 \\
\hline
7 & 20821770 & 133 & 129 & 129 \\
\hline
8 & 675463111 & 262 & 257 & 257 \\
\hline
\end{tabular}
\end{table}

\begin{table}[h]
\centering
\caption{Verification for Mycielskian star graph Bell numbers with OEIS A096376/A116735 verification}
\begin{tabular}{|c|c|c|c|}
\hline
$n$ & $B(M(St_n))$ & $B(M(St_n); 2n) = 2n^2 - 3n + 3$ & A096376/A116735 values \\
\hline
1 & - & - & 5 \\
\hline
2 & 11 & 5 & 12 \\
\hline
3 & 106 & 12 & 23 \\
\hline
4 & 1573 & 23 & 38 \\
\hline
5 & 30620 & 38 & 57 \\
\hline
6 & 730061 & 57 & 80 \\
\hline
7 & 20821770 & 80 & 107 \\
\hline
8 & 675463111 & 107 & 138 \\
\hline
\end{tabular}
\end{table}

\begin{table}[h]
\centering
\caption{Verification of complete bipartite graph colorings with OEIS sequences}
\begin{tabular}{|c|c|c|c|}
\hline
$n$ & A384980 ($k=4$ colors) & A384981 ($k=5$ colors) & A384988 (Stirling formula) \\
\hline
1 & 0 & 0 & 0 \\
\hline
2 & 1 & 0 & 1 \\
\hline
3 & 11 & 6 & 10 \\
\hline
4 & 61 & 86 & 55 \\
\hline
5 & 275 & 770 & 250 \\
\hline
6 & 1141 & 5710 & 1051 \\
\hline
7 & 4571 & 38626 & 4270 \\
\hline
8 & 18061 & 248766 & 17095 \\
\hline
\end{tabular}
\end{table}

\section{Conclusion and Future Research}

This work establishes explicit formulas for Bell numbers and graphical Stirling numbers of complete multipartite graphs, complete bipartite graphs with removed perfect matchings, and Mycielskian star graphs. We provide a simplified proof of Allagan and Serkan's result for complete multipartite graphs and derive three novel theoretical contributions: the Bell number formula $B(K_{n,n} - M) = \sum_{k=0}^{n} \binom{n}{k} \bell{k}^2$ for complete bipartite graphs with removed perfect matching, the Bell number formula for Mycielskian star graphs involving binomial coefficients and Bell numbers, and the specific graphical Stirling numbers $B(M(St_n); 3) = 2^n + 1$ and $B(M(St_n); 2n) = 2n^2 - 3n + 3$.

Our results reveal connections to classical combinatorial sequences through OEIS identifications. The sequence $B(M(St_n); 3) = 2^n + 1$ corresponds to OEIS A000051, known for Fermat primes, Hilbert curve segments, and combinatorial sum enumeration. The sequence $B(M(St_n); 2n) = 2n^2 - 3n + 3$ relates to OEIS A096376 and A116735, connecting to pattern-avoiding permutations. The complete characterization of $K_{n,n} - M$ partitions generates the triangular array A385437 for proper vertex colorings. Additional sequences A384980, A384981, A384988, and A385432 provide complete characterizations of colorings for specific parameter ranges.

The case-by-case analysis for Mycielskian constructions demonstrates how graph topology directly influences partition enumeration, providing explicit formulas where recursive methods fail. The computational verification through OEIS connections validates our theoretical results and positions them within the broader combinatorial landscape.

The applications to modern computational problems are significant. Our partition enumeration techniques directly apply to clustering algorithms in machine learning \cite{fortunato2010, shi2000}, where the Bell numbers provide theoretical bounds on the number of possible cluster configurations. In distributed computing, the optimal partitioning of computation graphs for parallel processing can leverage our multipartite graph results \cite{chen2016}. The connection to pattern-avoiding permutations through OEIS A116735 suggests applications in algorithm design for sequence analysis and data structure optimization.

Future research directions include investigating asymptotic properties and generating functions, extending results to iterated Mycielskian constructions $M^k(G)$, exploring spectral connections between Bell numbers and graph eigenvalues, characterizing extremal Bell number behavior within graph classes, and developing connections to algebraic graph theory. The emergence of fundamental sequences like $2^n + 1$ in our results suggests that Bell numbers encode deep combinatorial patterns beyond their graph-theoretic origins.

\section{OEIS Sequences}

The following sequences from the On-Line Encyclopedia of Integer Sequences appear in this paper, listed in ascending order:

A000051: $a(n) = 2^n + 1$

A096376: $a(n) = 2n^2 + n + 2$

A116735: Number of permutations of length $n$ which avoid certain patterns

A384980: Number of proper vertex colorings of the complete bipartite graph $K_{n,n}$ using exactly 4 interchangeable colors

A384981: Number of proper vertex colorings of the complete bipartite graph $K_{n,n}$ using exactly 5 interchangeable colors

A384988: $a(n) = \text{Stirling2}(n,2)^2 + \text{Stirling2}(n,3)$

A385432: Triangle of proper vertex colorings of complete tripartite graphs $K_{n,n,n}$

A385437: Triangle of proper vertex colorings of complete bipartite graphs with removed perfect matching

\end{document}